\documentclass[12pt]{amsproc}
\newtheorem{theorem}{\sc Theorem}[section]
\newtheorem{lemma}[theorem]{\sc Lemma}

\begin{document}
\title[BFC-theorems for approximate subgroups]{Bounded conjugacy classes, commutators, and approximate subgroups}

\author{Pavel Shumyatsky }
\address{ Pavel Shumyatsky: Department of Mathematics, University of Brasilia,
Brasilia-DF, 70910-900 Brazil}
\email{pavel@unb.br}
\thanks{This research was supported by the Conselho Nacional de Desenvolvimento Cient\'{\i}fico e Tecnol\'ogico (CNPq),  and Funda\c c\~ao de Apoio \`a Pesquisa do Distrito Federal (FAPDF), Brazil.}

\begin{abstract} Given a group $G$, we write $g^G$ for the conjugacy class of $G$ containing the element $g$. A famous theorem of B. H. Neumann states that if $G$ is a group in which all conjugacy classes are finite with bounded size, then the commutator subgroup $G'$ is finite. We establish the following results.

\noindent Let $K,n$ be positive integers and $G$ a group having a $K$-approximate subgroup $A$. 

\noindent If $|a^G|\leq n$ for each $a\in A$, then the commutator subgroup of $\langle A^G\rangle$ has finite $(K,n)$-bounded order.

\noindent If $|[g,a]^G|\leq n$ for all $g\in G$ and $a\in A$, then the commutator subgroup of $[G,A]$ has finite $(K,n)$-bounded order. 
\end{abstract}

\maketitle

\section{Introduction} Let $K$ be a positive integer. A symmetric subset $A$ of a group $G$ is a $K$-approximate subgroup if there is a finite subset $E\subseteq G$ such that $|E|\leq K$ and $AA\subseteq EA$. The formal definition of an approximate subgroup was introduced by Tao in \cite{tao}. Since then many important results on approximate subgroups have been established. In particular, Breuillard, Green and Tao essentially described the structure of finite approximate subgroups \cite{bgt}. The reader is referred to the recent book \cite{tointon}, or the surveys \cite{breu1,breu2}, for detailed information on these developments. 

An interesting principle was stated in \cite{breu2}:\bigskip

 {\it Group-theoretical arguments can often be successfully transferred to approximate subgroups}.
\bigskip

In the present article we check this principle against certain variations of B. H. Neumann's theorem that a BFC-group has finite commutator subgroup.

Given a group $G$ and an element $x\in G$, we write $x^G$ for the conjugacy class containing $x$. More generally, if $X,Y\subseteq G$, we write $X^Y$ for the set of all $x^y$, where $x\in X$ and $y\in Y$. Of course, if the number of elements in $x^G$ is finite, we have $|x^G|=[G:C_G(x)]$. A group is called a BFC-group if its conjugacy classes are finite and have bounded size. In 1954 B. H. Neumann discovered that the commutator subgroup $G'$ of a BFC-group $G$ is finite \cite{bhn}. It follows that if $|x^G|\leq n$ for each $x\in G$, then $G'$ has finite $n$-bounded order. Throughout the article we use the expression ``$(a,b,\dots)$-bounded" to mean that a quantity is finite and bounded by a certain number depending only on the parameters $a,b,\dots$. A first explicit bound for the order of $G'$ was found by J. Wiegold \cite{wie}, and the best known was obtained in \cite{gumaroti} (see also \cite{neuvoe} and \cite{sesha}). The article \cite{dieshu} deals with groups $G$ in which conjugacy classes containing commutators are bounded. It is shown that if $|x^G|\leq n$ for any commutator $x$, then the second commutator subgroup $G''$ has finite $n$-bounded order. Later this was extended in \cite{dms} to higher commutator subgroups. A related result for groups in which the conjugacy classes containing squares have finite bounded sizes was obtained in \cite{squares}.

A stronger version of the Neumann theorem was recently established in \cite{ashu}:
\bigskip

\noindent {\it Let $n$ be a positive integer and $G$ a group having a subgroup $A$ such that $|a^G|\leq n$ for each $a\in A$. Then the commutator subgroup of $\langle A^G\rangle$ has finite $n$-bounded order.}
\bigskip

Here, as usual, $\langle X\rangle$ denotes the subgroup generated by the set $X$ and so $\langle A^G\rangle$ denotes the minimal normal subgroup containing $A$.

In the present paper we extend the above result as follows.

\begin{theorem}\label{main1} Let $K,n$ be positive integers and $G$ a group having a $K$-approximate subgroup $A$ such that $|a^G|\leq n$ for each $a\in A$. Then the commutator subgroup of $\langle A^G\rangle$ has finite $(K,n)$-bounded order.
\end{theorem}

For a subset $X$ of a group $G$ we write $[G,X]$ to denote the subgroup generated by all commutators $[g,x]$, where $g\in G$ and $x\in X$. It is well-known that $[G,X]$ is a normal subgroup of $G$. Moreover, $[G,X]=[G,\langle X\rangle]$. We examine approximate subgroups $A\subseteq G$ such that the conjugacy classes of commutators $[g,a]$ have bounded sizes whenever $g\in G$ and $a\in A$.

\begin{theorem}\label{main2} Let $K,n$ be positive integers and $G$ a group having a $K$-approximate subgroup $A$ such that $|[g,a]^G|\leq n$ for all $g\in G$ and $a\in A$. Then the commutator subgroup of $[G,A]$ has finite $(K,n)$-bounded order.
\end{theorem}

It is worthwhile to mention that Theorem \ref{main2} was unknown even in the case where $A$ is a subgroup of $G$. It can be regarded as an extension of the aforementioned result in \cite{dieshu} that if $|x^G|\leq n$ for any commutator $x$, then the second commutator subgroup $G''$ has finite $n$-bounded order.

\section{Preliminaries} 

Let $G$ be a group generated by a set $X$ such that $X = X^{-1}$. Given an element $g\in G$, we write $l_X(g)$ for the minimal number $l$ with the property that $g$ can be written as a product of $l$ elements of $X$. A proof of the following result can be found in \cite[Lemma 2.1]{dieshu}.

\begin{lemma}\label{21} Let $G$ be a group generated by a set $X=X^{-1}$ and let $L$ be a subgroup of finite index $m$ in $G$. Then each coset $Lb$ contains an element $g$ such that $l_X(g)\leq m-1$.
\end{lemma}

The next lemma is almost obvious.

\begin{lemma}\label{22} Let $k,n,s\geq1$, and let $G$ be a group containing a set $X=X^{-1}$ such that $|x^G|\leq n$ for any $x\in X$. Let $g_1,\dots,g_s\in\langle X\rangle$ and assume that $l_X(g_i)\leq k$. Then $C_G(g_1,\dots,g_s)$ has finite $(k,n,s)$-bounded index in $G$.
\end{lemma}
\begin{proof} Since $l_X(g_i)\leq k$, we can write $g_i=x_{i1}\ldots x_{ik}$, where $x_{ij}\in X$ and $i=1,\ldots,s$. By the hypothesis the index of $C_G(x_{ij})$ in $G$ is at most $n$ for any such element $x_{ij}$. Set $U=\cap_{i,j}C_G(x_{ij})$. We have $[G:U]\leq n^{ks}$. Since $U\leq C_G(g_1,\dots,g_s)$, the lemma follows.
\end{proof}

The next observation will play a crucial role in the proof of Theorems \ref{main1} and \ref{main2}.

\begin{lemma}\label{23} Let $X$ be a normal subset of a group $G$ such that $|x^G|\leq n$ for any $x\in X$, and let $H=\langle X\rangle$. Then the subgroup $\langle[H,x]^G\rangle$ has finite $n$-bounded order.
\end{lemma}
\begin{proof} Without loss of generality we can assume that $X=X^{-1}$. Let $m$ be the maximum of indices of $C_H(x)$ in $H$ for $x\in X$. Of course, $m\leq n$. Take $x\in X$. Since the index of $C_H(x)$ in $H$ is at most $m$, by Lemma \ref{21} we can choose elements $y_1,\ldots,y_m$ in $H$ such that $l_X(y_i)\leq m-1$ and the subgroup $[H,x]$ is generated by the commutators $[y_i,x]$, for $i=1,\ldots,m$.  For any such $i$  write $y_i=y_{i1}\ldots y_{i(m-1)}$, with $y_{ij}\in X$. By using standard commutator identities we can rewrite $[y_i,x]$ as a product of conjugates in $H$ of the commutators $[y_{ij},x]$. Let $\{h_1,\ldots,h_s\}$ be the conjugates in $H$  of all elements from the set $\{x,y_{ij} \mid  1\leq i\leq m,\ 1\leq j\leq m-1\}.$ Note that the number $s$ here is $m$-bounded. This follows form the fact that $C_H(x)$ has index at most $m$ in $H$ for each $x\in X$.  Put $T=\langle h_1,\ldots,h_s \rangle$. Observe that the centre $Z(T)$ has index at most $m^s$ in $T$, since the index of $C_H(h_i)$ in $H$ is at most $m$ for any $i=1,\ldots,s$.  Thus, by Schur's theorem \cite[10.1.4]{Rob}, we conclude that the commutator subgroup $T'$ has finite $m$-bounded order. Since $[H,x]$ is contained in $T'$, deduce that the order of $[H,x]$ is $m$-bounded. Further, the subgroup $[H,x]$ is normal in $H$ and there are at most $n$ conjugates of $[H,x]$ in $G$. Therefore $\langle[H,x]^G\rangle$ is a product of at most $n$ normal subgroups, each of which has $n$-bounded order. Hence, the result.
\end{proof}

\begin{lemma} \label{uxx} Let $G$ be a group and $x,y\in G$. Assume that $|x^G|=m$ and $|(yx)^G|\leq m$. Suppose that there are $b_1,\dots,b_m\in G$ such that $x^G=\{x^{b_1},\dots,x^{b_m}\}$ and $y\in C_G(b_1,\dots,b_m)$. Then $[G,y]\leq[G,x]$.
\end{lemma}
\begin{proof} First we note that $(yx)^G=\{yx^{b_1},\dots,yx^{b_m}\}$. This is because all elements $yx^{b_1},\dots,yx^{b_m}$ are different and since by the hypothesis $(yx)^G$ contains at most $m$ elements, the class $(yx)^G$ must coinside with $\{yx^{b_1},\dots,yx^{b_m}\}$.  Therefore for any $g\in G$ there is $b_i\in\{b_1,\dots,b_m\}$ such that $(yx)^g=yx^{b_i}$. So $y^gx^g=yx^{b_i}$ and $[y,g]=x^{b_i}x^{-g}\in[G,x]$. The lemma follows.
\end{proof}

\section{Proof of Theorem \ref{main1}}

Recall that $A$ is a $K$-approximate subgroup of a group $G$ such that $|a^G|\leq n$ for any $a\in A$. We wish to show that $H=\langle A^G\rangle$ has finite commutator subgroup of $n$-bounded order. 

Let $E=\{e_1,\dots,e_K\}$ be a set of size $K$ such that $AA\subseteq EA$. It will be assumed that $K$ is chosen as small as possible and so for each $i=1,\dots,K$ there are $x_{i1},x_{i2},x_{i3}\in A$ such that $x_{i1}x_{i2}=e_ix_{i3}$. By Lemma \ref{23} each of the subgroups $\langle[H,x_{ij}]^G\rangle$ has $n$-bounded order. Hence, also their product has finite $(K,n)$-bounded order. Now we can pass to the quotient over the product of all $\langle[H,x_{ij}]^G\rangle$ and without loss of generality assume that the set $E$ is contained in the centre of $H$. 

Denote by $X$ the set $A^G$. Let $m$ be the maximum of indices of $C_H(x)$ in $H$ for $x\in X$. Of course, $m\leq n$.
Select $a\in A$ such that $|a^H|=m$. Choose $b_1,\ldots,b_m$ in $H$ such that $l_X(b_i)\leq m-1$ and $a^H=\{a^{b_i};i=1,\ldots,m\}$. The existence of the elements $b_i$ is guaranteed by Lemma \ref{21}. Set  $U=C_G(\langle b_1,\ldots,b_m \rangle)$. In view of Lemma \ref{22} note that the index of $U$ in $G$ is $n$-bounded. 

Let $r$ be the minimal number for which there are elements $d_1,\dots,d_r\in A$ such that $A$ is contained in the union of the left cosets $d_iU$. We fix the elements $d_i$ and denote by $S$ the product of the subgroups $\langle[H,d_i]^G\rangle$ for $i=1,\dots,r$. By Lemma \ref{23} $S$ is a product of at most $r$ normal subgroups of finite $n$-bounded order. Taking into account that $r$ is $n$-bounded conclude that $S$ has finite $n$-bounded order.

Choose any $u\in A\cap U$. Since $AA\subseteq EA$ write $ua=ex$ for suitable $e\in E$ and $x\in A$. Since $e\in Z(H)$ and $|x^H|\leq m$, it follows that $|(ua)^H|\leq m$. Recall that $U=C_G(\langle b_1,\ldots,b_m \rangle)$. Lemma \ref{uxx} implies that $[H,u]\leq[H,a]$. This happens for every choice of $u\in A\cap U$ and so $[H,(A\cap U)]\leq[H,a]$.

Let $T=\langle[H,a]^G\rangle$ and observe that by virtue of Lemma \ref{23} $T$ has finite $n$-bounded order. Choose an arbitrary element $d\in A$. There is an index $j$ such that $d\in d_jU$ and ${d_j}^{-1}d\in U$. Since ${d_j}^{-1}d\in AA$, write ${d_j}^{-1}d=ey$ for suitable $e\in E$ and $y\in A$. Observe that $e\in U$, whence $y\in U$. Taking into account that $[H,d_j]\leq S$, $[H,e]=1$, and $[H,y]\leq T$ deduce $$[H,d]=[H,d_jey]\leq[H,d_j][H,e][H,y]\leq ST.$$

Since $d$ was chosen in $A$ arbitrarily, conclude that $[H,A]\leq ST$. Recall that $H=\langle A^G\rangle$. We therefore conclude that $H'\leq ST$. Now the theorem follows from the fact that the order of $ST$ is $n$-bounded. This completes the proof of the theorem. 

\section{Proof of Theorem \ref{main2}} 

Throughout this section $A\subseteq G$ is a $K$-approximate subgroup of a group $G$ such that $|[g,a]^G|\leq n$ for each $g\in G$ and $a\in A$. We need to prove that $[G,A]$ has finite commutator subgroup of $(K,n)$-bounded order.

Let $X$ be the set of all conjugates of commutators $[g,a]$, where $g\in G$ and $a\in A$. Note that the set $X$ is symmetric. Put $H=\langle X\rangle$. By Lemma \ref{23} the subgroup  $\langle[H,x]^G\rangle$ has finite $n$-bounded order whenever $x\in X$. 
\begin{lemma} \label{one} For any $a\in A$ the subgroup $[H,[G,a]]$ has finite $n$-bounded order.
\end{lemma}
\begin{proof} Choose $a\in A$. Let $m_0$ be the maximum of indices of $C_H(x)$ in $H$, where $x$ ranges through the set of commutators $[g,a]$ with $g\in G$. Select $g_0\in G$ such that $|[g_0,a]^H|=m_0$. Choose   $b_1,\ldots,b_{m_0}$ in $H$ such that $l_X(b_i)\leq m_0-1$ and $[g_0,a]^H=\{[g_0,a]^{b_i}; i=1,\ldots,m_0\}$. (The existence of the elements $b_i$ is guaranteed by Lemma \ref{21}.) Set  $U=C_G(\langle b_1,\ldots,b_{m_0} \rangle)$. Note that by Lemma \ref{22} the index of $U$ in $G$ is $n$-bounded. Let $U_0=\cap_{g\in G} U^g$ be the maximal normal subgroup of $G$ contained in $U$. Obviously, the index of $U_0$ in $G$ is $n$-bounded as well. For any $g\in G$ observe that $[gg_0,a]=[g,a]^{g_0}[g_0,a]$. Choose $g\in U_0$ and set $[g,a]^{g_0}=u$. Lemma \ref{uxx} shows that $[H,u]\leq[H,[g_0,a]]$. 

Let $c_1,\dots,c_k$ be a transversal of $U_0$ in $G$. For $i=1,\dots,k$ let $T_i$ denote the subgroup $\langle[H,[c_i,a]]^G\rangle$. In view of Lemma \ref{23} each subgroup $T_i$ has finite $n$-bounded order. Further, let $T_0$ denote the subgroup $\langle[H,[g_0,a]]^G\rangle$. Likewise, $T_0$ has finite $n$-bounded order. Let $N$ be the product of all $T_i$ for $i=0,1,\dots,k$.

Any element $g\in G$ can be written as a product $g=xc_j$ for suitable $x\in U_0$ and $j\leq k$. Then we have $[g,a]=[xc_j,a]=[x,a]^{c_j}[c_j,a]$. We now know that the images in $G/N$ of both $[x,a]$ and $[c_j,a]$ are central in $H/N$. It follows that also the image $[G,a]$ is central in $H/N$. Since $N$ has finite $n$-bounded order, the lemma follows.
\end{proof} 

Let $E=\{e_1,\dots,e_K\}$ be a set of size $K$ such that $AA\subseteq EA$. It will be assumed that $K$ is chosen as small as possible and so for each $i=1,\dots,K$ there are $x_{i1},x_{i2},x_{i3}\in A$ such that $x_{i1}x_{i2}=e_ix_{i3}$. By Lemma \ref{one} for each $x_{ij}$ the subgroup $N_{ij}=[H,[G,x_{ij}]]$ has finite $n$-bounded order. Let $N$ be the product of all these subgroups $N_{ij}$ and observe that $N$ has finite $(K,n)$-bounded order. Pass to the quotient $G/N$ and assume that $[H,[G,x_{ij}]]=1$ for all $i,j$. Then of course $[H,[G,e_i]]=1$ for all $i=1,\dots,K$. Therefore in what follows, without loss of generality, we will assume that $$[G,E]\leq Z(H).$$

Let $m$ be the maximum of indices of $C_H(x)$ in $H$, where $x$ ranges through the set $X$.

\begin{lemma}\label{three} For any $b\in\langle A\rangle$ and $g\in G$ we have $|[g,b]^H|\leq m$.  
\end{lemma}
\begin{proof} Indeed, since $b\in\langle A\rangle$, we can write $b=ea$ for suitable $a\in A$ and $e\in\langle E\rangle$. Then $[g,b]=[g,ea]\in[G,E]X$. Now use that $[G,E]\leq Z(H)$ and $|x^H|\leq m$ for any $x\in X$ and deduce the lemma.
\end{proof}

Now fix $h_0\in G$ and $a_0\in A$ such that $|[h_0,a_0]^H|=m$. Choose $h_1,\ldots,h_{m}$ in $H$ such that $l_X(h_i)\leq m-1$ and $$[h_0,a_0]^H=\{[h_0,a_0]^{h_i};i=1,\ldots,m\}.$$ The existence of the elements $h_i$ follows from Lemma \ref{21}. Set $V=C_G(\langle h_1,\ldots,h_m\rangle)$. Note that by Lemma \ref{22} the index of $V$ in $G$ is $n$-bounded. Let $V_0=\cap_{g\in G}V^g$ be the maximal normal subgroup of $G$ contained in $V$ and note that the index of $V_0$ in $G$ is $n$-bounded as well. By Lemma \ref{one} the subgroup $S=[[H,[G,a_0]]$ has finite $n$-bounded order. 

\begin{lemma}\label{two}
  $[H,[G,V_0\cap\langle A\rangle]]\leq S$.
\end{lemma}
\begin{proof} Let $b\in V_0\cap\langle A\rangle$. Lemma \ref{three} tells us that $|[g,a_0b]^H|\leq m$ for any $g\in G$. Moreover, observe that $[g,a_0b]=[g,b][g,a_0]^b$ while $[g,b]\in V_0$. Lemma \ref{uxx} shows that $[H,[g,b]]\leq [H,[g,a_0]]\leq S$. This happens for every $g\in G$ so the lemma follows.
\end{proof}

Let $r$ be the minimal number for which there are elements $d_1,\dots,d_r\in A$ such that $A$ is contained in the union of the left cosets $d_1V_0,\dots,d_rV_0$. We fix the elements $d_i$ and for each $i=1,\dots,r$ put $M_i=[H,[G,d_i]]$. By Lemma \ref{one} the product of all $M_i$ has finite $n$-bounded order. Pass to the quotient $G/\prod_iM_i$ and, without loss of generality, assume that $[G,d_i]\leq Z(H)$ for each $i$.

For an arbitrary element $a\in A$ there is $i\leq r$ such that $a\in d_iV_0$. We have $[G,a]\leq[G,d_i][G,{d_i}^{-1}a]$. By assumptions, $[H,[G,d_i]]=1$. Taking into account that ${d_i}^{-1}a\in V_0\cap\langle A\rangle$ and using Lemma \ref{two} deduce that $[H,[G,{d_i}^{-1}a]]\leq S$. Thus, $[H,[G,a]]\leq S$ whenever $a\in A$. Since $H=\prod_{a\in A}[G,a]$, it follows that $H'\leq S$. This completes the proof of the theorem.


\begin{thebibliography}{10}
\bibitem{ashu} C. Acciarri, P. Shumyatsky, A stronger form of Neumann's BFC-theorem, Israel J. Math., to appear, https://arxiv.org/pdf/2003.09933.pdf
\bibitem{breu1} E. Breuillard, Lectures on approximate groups, Lecture Notes, https://www.math.u-psud.fr/~breuilla/ClermontLectures.pdf
\bibitem{breu2} E. Breuillard, A brief introduction to approximate groups, Thin groups and superstrong approximation, Math.Sci. Res. Inst. Publ., {\bf 61}, Cambridge Univ. Press, Cambridge, 2014, pp. 23--50.
\bibitem{bgt} E. Breuillard, B. Green and T. Tao, The structure of approximate groups, Publ. Math. IHES, {\bf 116} (2012), 115--221.
\bibitem{dms}  E. Detomi, M. Morigi, P. Shumyatsky, BFC-theorems for higher commutator subgroups, Quarterly J. Math. {\bf 70} (2019), no. 3, 849--858.
\bibitem{dieshu} G. Dierings, P. Shumyatsky, Groups with boundedly finite conjugacy classes of commutators, Quarterly J. Math.  {\bf 69} (2018), no.~3, 1047--1051.
\bibitem{squares} G. Dierings, P. Shumyatsky, Groups in which squares have boundedly many conjugates, J. Group Theory {\bf 22} (2019), 133--136.
\bibitem{gumaroti} R. M. Guralnick, A. Maroti, Average dimension of fixed point spaces with applications, J. Algebra {\bf 226} (2011), 298--308.
\bibitem{bhn} B. H. Neumann, Groups covered by permutable subsets, J. London Math. Soc. (3) {\bf 29} (1954), 236--248.
\bibitem{neuvoe} P. M. Neumann, M.R. Vaughan-Lee, An essay on BFC groups, Proc. Lond. Math. Soc.  {\bf 35} (1977), 213--237.
\bibitem{Rob}  D. J. S. Robinson,  {\it A course in the theory of groups},  Second edition. Graduate Texts in Mathematics, 80. 
Springer-Verlag, New York, 1996.
\bibitem{sesha} D. Segal, A. Shalev, On groups with bounded conjugacy classes, Quart. J. Math. Oxford {\bf 50} (1999), 505--516.
\bibitem{tao} T. Tao, Product set estimates for non-commutative groups, Combinatorica, {\bf 28} (2008), 547--594.
\bibitem{tointon} M. C. H. Tointon, Introduction to approximate groups, Cambridge University Press, Cambridge, 2020. 
\bibitem{wie} J. Wiegold, Groups with boundedly finite classes of conjugate elements, Proc. Roy. Soc. London Ser. A {\bf 238} (1957), 389--401.
\end{thebibliography}
\end{document}